\documentclass[joc]{ipart}
\usepackage{amsfonts,amsmath,amsxtra,amsthm,tikz}
\usepackage{mathrsfs}
\usepackage{times}
\usepackage{anysize}
\usepackage{float}
\usepackage{stackengine}
\usepackage{amscd}
\usepackage{amsthm}
\usepackage{comment}
\usepackage{amssymb}
\usepackage{color}
\usepackage{soul}
\usepackage{hyperref}

\usetikzlibrary{arrows}
\usetikzlibrary{arrows.meta}

\startlocaldefs
\newcommand{\pw}[1]{\langle #1 \rangle} % wrap list notation for parking functions
\newcommand{\ZZ}{\mathbb{Z}_{\ge 0}}

\newtheorem{theorem}{Theorem}%[section]
\newtheorem{lemma}[theorem]{Lemma}%[section]
\newtheorem{example}[theorem]{Example}%[section]

\newtheorem{conjecture}[theorem]{Conjecture}
\theoremstyle{definition}
\newtheorem{definition}[theorem]{Definition}%[section]
\newtheorem{remark}[theorem]{Remark}
\endlocaldefs

\pubyear{2019}
\volume{0}
\issue{0}
\firstpage{1}
\lastpage{1}
\begin{document}

\begin{frontmatter}

%%  "Title of the Paper"
\title{Pak-Stanley labeling for Central Graphical Arrangements.}
%\thankstext{T1}{???}
%\protect\thanksref{T1}
\begin{aug}
  \author{\fnms{Mikhail } \snm{Mazin}\thanksref{t2}\ead[label=e1]{mmazin@math.ksu.edu}}
  \thankstext{t2}{The author is supported by the Simons Foundation Collaboration Grant 524324}
  \address{Mathematics Department, Kansas State University.
\\ Cardwell Hall, Manhattan, KS 66506\\ \printead{e1}}
%\\\printead{e2}
\author{\fnms{Joshua} \snm{Miller}\ead[label=e4]{jgm019@ksu.edu}}
\address{Mathematics Department, Kansas State University.
\\ Cardwell Hall, Manhattan, KS 66506\\\printead{e4}}

\end{aug}
%%  History:
%\received{\sday{3} \smonth{1} \syear{2019}}

\begin{abstract}
The original Pak-Stanley labeling was defined by Pak and Stanley as a bijective map from the set of regions of an extended Shi arrangement to the set of parking functions. This map was later generalized to other arrangements associated with graphs and directed multigraphs . In these more general cases the map is no longer bijective. However, it was shown that it is surjective to the set of the $G$-parking functions, where $G$ is the multigraph associated with the arrangement.

This leads to a natural question:  when is the generalized Pak-Stanley map bijective? In this paper we answer this question in the special case of central hyperplane arrangements, i.e. the case when all the hyperplanes of the arrangement pass through a common point.
%Back in the nineties Pak and Stanley introduced a labeling of the regions of a $k$-Shi arrangement by $k$-parking functions and proved its bijectivity. Duval, Klivans, and Martin considered a modification of this construction associated with a graph $G.$ They introduced the $G$-Shi arrangement and a labeling of its regions by $G$-parking functions. They conjectured that their labeling is surjective, i.e. that every $G$-parking function appears as a label of a region of the $G$-Shi arrangement. Later Hopkins and Perkinson proved this conjecture. In particular, this provided a new proof of the bijectivity of Pak-Stanley labeling in the $k=1$ case. We generalize Hopkins-Perkinson's construction to the case of arrangements associated with oriented multigraphs. In particular, our construction provides a simple straightforward proof of the bijectivity of the original Pak-Stanley labeling for arbitrary $k.$
\end{abstract}


%%  Upper case for every keyword

%\tableofcontents

\end{frontmatter}

%%  The body
\section*{Introduction.}

Let $V\subset \mathbb{R}^n$ be given by $x_1+\ldots+x_n=0.$ Consider an arrangement $\mathcal{A}$ of affine hyperplanes in $V,$ such that every hyperplane of $\mathcal{A}$ is of the form $H_{i,j}^a:=\{x_i-x_j=a\}$ for some $i,j\in\{1,\ldots n\}$ and $a>0.$ Let $G_\mathcal{A}$ be the associated directed multigraph, defined as follows. The set of vertices of $G_\mathcal{A}$ is $\{1,\ldots,n\},$ and the edge $i\to j$ has multiplicity 

$$
m_{ij}:=\#\{a\in\mathbb{R}_{>0}|H_{i,j}^a\in\mathcal{A}\}.
$$

Note that one gets $m_{ij}+m_{ji}$ hyperplanes parallel to $\{x_i=x_j\}$ in $\mathcal{A},$ $m_{ij}$ of them on one side of the origin, and $m_{ji}$ of them on the other. Note also that the multigraph $\mathcal{G}_\mathcal{A}$ does not determine the combinatorial type of the arrangement $\mathcal{A},$ as one can shift the hyperplanes by changing the constants on the right hand sides of the equations without changing the graph. 

\begin{definition}
We will call the arrangements of the type described above the \textit{multigraphical arrangements.}
\end{definition}

The generalized Pak-Stanley labeling of the regions (connected components of the complement) of a multigraphical arrangement was defined in \cite{M17}:

\begin{definition}
Let $R$ be a region of $\mathcal{A}.$ Let $\mathcal{A}_R\subset \mathcal{A}$ be the subset consisting of the hyperplanes that separate $R$ from the origin.  We define the label $\lambda_R$ to be the function $\lambda_R:\{1,\ldots,n\}\to\mathbb{Z}_{\ge 0}$ given by the following formula: 

$$
\lambda_R(i):=\#\{(a,j)|a\in \mathbb{R}_{>0}, j\in\{1,\ldots,n\}, \text{and}\  H^a_{i,j}\in\mathcal{A}_R\}.
$$
In other words, $\lambda_R(i)$ equals to the number of hyperplanes of the arrangement $\mathcal{A}$ of the form $H_{i,j}^a$ separating the region $R$ from the origin. (Note that here $i$ is fixed, but $j$ and $a$ might vary.) 
\end{definition}

We will use the notation $\pw{\lambda(1),\ldots,\lambda(n)}$ for a label $\lambda.$ The region $R_0$ containing the origin is called the \textit{fundamental region.} It is the only region labeled by $\pw{0,\ldots,0}.$ Note that the labeling can be defined inductively: as one crosses a hyperplane $H_{ij}^a=\{x_i-x_j=a>0\}$ in the direction \textbf{away} from the origin, the $i$th  component of the label is increased by one, while the rest of the components remain unchanged.

\begin{definition}
Let $G$ be a directed multigraph on a vertex set $\{1,\ldots,n\}.$ A function $\lambda:\{1,\ldots,n\}\to\ZZ$ is called a {\it $G$-parking function} if for any non-empty subset $I\subset\{1,\ldots,n\}$ there exists a vertex $i\in I$ such that the number of edges $(i\to j)\in E_G$, counted with multiplicity, such that $j\notin I$ is greater than or equal to $\lambda(i).$
\end{definition}

The following results were proved in \cite{HP12} and \cite{M17}:

\begin{theorem}[\cite{HP12,M17}]
Let $R$ be any region of a multigraphical arrangement $\mathcal{A}.$ Then the corresponding label $\lambda_R$ is a $\mathcal{G}_\mathcal{A}$-parking function.
\end{theorem}

\begin{theorem}[\cite{HP12,M17}]
Let $\mathcal{A}$ be a multigraphical arrangement, and let $\lambda$ be any $\mathcal{G}_\mathcal{A}$-parking function. Then there exists a region $R$ of $\mathcal{A},$ such that $\lambda_R=\lambda.$
\end{theorem}

Combining the above, we get that the generalized Pak-Stanley labeling is a  surjective map from the set of regions of $\mathcal{A}$ to the set of $\mathcal{G}_\mathcal{A}$-parking functions. 

In \cite{HP12} these results were proved in a more restricted context. In \cite{M17} they were generalized to multigraphical arrangements. In the classical case of extended Shi arrangements, one can show the bijectivity of the Pak-Stanley labeling by using the above results and then comparing the cardinalities of the two sets. The bijectivity results can be extended to other families of arrangements (see \cite{DO18}). However, in general the generalized Pak-Stanley labelings  often fail to be injective. Study of the examples suggests that whenever the map is not injective "globally" it is also not injective 'locally."

\begin{definition}\label{local and global}
Let $\mathcal{A}$ be a multigraphical arrangement and let $p\in V$ be any point. We say that the Pak-Stanley labeling for $\mathcal{A}$ is locally injective near $p$ if all labels of regions $R$ such that $p\in\overline{R}$ are distinct. Further, if this holds for all  $p\in V$, we say that $\mathcal{A}$ is locally injective. 
\end{definition}

\begin{conjecture}\label{Conjecture: inject vs local inject}
Let $\mathcal{A}$ be a multigraphical arrangement, then the generalized Pak-Stanley map from the set of regions of $\mathcal{A}$ to the set of parking functions is injective if and only if it is injective locally. 
\end{conjecture}

\begin{remark}
Note that the ``only if'' part of the conjecture is trivial. 
\end{remark}

Furthermore, the examples indicate that a stronger form of Conjecture \ref{Conjecture: inject vs local inject} can be made about the proximity of repeated labels.

\begin{conjecture}\label{conjecture: global}
Let $\mathcal{A}$ be a multigraphical arrangement. Then for a fixed label $\lambda$, the closure of the union of all regions labeled by $\lambda$ is connected.

%and assume that the generalized Pak-Stanley map from the set of parking functiosn to the set of regions of $\mathcal{A}$ is not injective
\end{conjecture}

It is clear from the definition of the Pak-Stanley labeling that the local injectivity near a point $x\in V$ is a local question. More precisely one has the following fact

\begin{lemma}\label{observation: local inject}
Let $\mathcal{A}$ be a multigraphical arrangement and $p\in V$ be a point. Let $\mathcal{A}_p\subset \mathcal{A}$ be the subarrangement consistenting of hyperplanes that contain $p$. Then the Pak-Stanley labeling for $\mathcal{A}$ is locally injective near $p$ if and only if the Pak-Stanley labeling for $\mathcal{A}_p$ is injective. 
\end{lemma}
\begin{proof}
Let $R_0, \dots , R_N$ be all of the regions of $\mathcal{A}$ whose closures contain the point $p$. Let also $R^p_0, \dots, R^p_N$ be the corresponding regions of $\mathcal{A}_p$, i.e. $R_i\subset R^p_i$ for each $i$. Let $\lambda(R_i)$ denote the Pak-Stanley labeling for $\mathcal{A}$ and $\lambda_p(R^p_i)$ denote the Pak-Stanley labeling for $\mathcal{A}_p.$ Let us also assume that the origin belongs to the region $R^p_0,$ so that $\lambda_p(R^p_0)=\pw{0,\ldots,0}.$ According to the inductive definition of the labeling, as one crosses a hyperplane of $\mathcal{A}_p,$ the labels for $\mathcal{A}$ and $\mathcal{A}_p$ change in the same way. Therefore, one gets 
\begin{equation*}
\lambda(R_i)=\lambda_p(R^p_i)+\lambda(R_0)
\end{equation*}
for all $i.$ It follows that $\lambda(R_i)=\lambda(R_j)$ if and only if $\lambda_p(R^p_i)=\lambda_p(R^p_j),$ which concludes the proof. 

%By using the inductive definition of labeling, one immediately sees that for each $i$ that $\lambda(R_i)=\lambda_p(R^p_i)+\lambda(R_0)$, where $\lambda$ is the labeling for $\mathcal{A}$ and $\lambda_p$ is the labeling for $\mathcal{A}_p.$
\end{proof}

The natural question is to characterize the directed multigraphs for which there exist arrangements with bijective labelings. % In this paper we answer this question for the special case of central affine multigraphical arrangements. 
Conjecture \ref{Conjecture: inject vs local inject} and Observation \ref{observation: local inject} motivate studying this question in the special case of central hyperplane arrangements, i.e. arrangements for which all the hyperplanes pass through a common point.% since for a multigraphical arrangement $\mathcal{A}$, locally injectivity at the point $x\in V$ is akin to checking if the central multigraphical arrangement $\mathcal{A}_x$ is injective. Moreover, Conjecture \ref{conjecture: global} hints that if an arrangement $\mathcal{A}$ is locally injective for all $x\in V$, then it is globally injective.

In this paper we answer this question for the special case of central affine multigraphical arrangements, which correspond to acyclic digraphs, by giving necessary and sufficient conditions on the digraph such that the labeling is injective. 

Baker, in \cite{BB19}, has been working on generalizing to arbitrary multigraphical arrangements in the $n=3$ case. She noticed that arrangements with a bijective labeling had a corresponding graph that satisfied the following.
\begin{theorem}\label{bethany}
Suppose $\mathcal{A}$ is a multigrahical arrangment with a bijective Pak-Stanley labeling and $G_{\mathcal{A}}$ is the corresponding graph, then for $i,j,k\in V$ with $m_{ij}\neq 0$, $m_{ik}\neq 0$, then $m_{jk}+m_{kj}\ge m_{ij}+m_{ik}-1$.
\end{theorem}

In her thesis, she showed that the above criterion was necessary, but not sufficient for a graph to yield a bijective labeling by giving several families of graphs that satisfy the conditions of Theorem \ref{bethany} but do not emit arrangements with a bijective labeling. 

If $\mathcal{A}$ is a central multigraphical arrangement, then $G_{\mathcal{A}}$ is a simple acyclic digraph, and the condition of Theorem \ref{bethany} reduces to the following: if both edges $i\to j$ and $i\to k$ are in $G_{\mathcal{A}}$ then $m_{kj}+m_{jk}>1,$ i.e. either $j\to k$ or $k\to j$ is also in $G_{\mathcal{A}}.$ The main result of this paper is that in this case the condition is not only necessary for the bijectivity of the Pak-Stanley labeling of $\mathcal{A}$, but also sufficient  (see Theorem \ref{Theorem: Main}).

%\begin{remark}
%For any simple acyclic graph, then  by Theorem \ref{bethany} we have the condition presented in Theorem \ref{Theorem: Main}. If $m_{ij}=m_{ik}=1$, i.e. we have both edges $i\to j$ and $i\to k$, then we have that $m_{kj}+m_{jk}>1,$ i.e. either $j\to k$ or $k\to j$ is also in the graph. However, in the case of central multigraphical arrangements the condition is not only necessary, but sufficient which will be proved in Theorem \ref{Theorem: Main}.  
%\end{remark}

\begin{figure} 
\begin{center}
\begin{tikzpicture}[scale=.75]
\draw (-10,-1.732) node {\small $1$};
\draw (-7.5,1.732) node {\small $2$};
\draw (-5,-1.732) node {\small $3$};

\draw [-{Latex[length=2mm, width=2mm]}] (-9.8,-1.732)--(-5.2,-1.732);
\draw [-{Latex[length=2mm, width=2mm]}] (-9.8,-1.632)--(-7.6,1.532);
\draw [-{Latex[length=2mm, width=2mm]}] (-7.7,1.732)--(-10,.-1.532);
\draw [-{Latex[length=2mm, width=2mm]}] (-7.3,1.732)--(-5,-1.532);

\draw (-2,-3.464)--(2,3.464);
\draw (-4.,-3.464)--(0,3.464);
\draw (-2,3.464)--(2,-3.464);
\draw (-4,0)--(2,0);
\draw (-1.5,-.75) node{\small $\pw{0,0,0}$};
\filldraw [fill=black] (-1.5,-.25) circle [radius=0.1];
\draw (0,-1.5) node{\small $\pw{1,0,0}$};
\draw (1.5,-.75) node{\small $\pw{2,0,0}$};
\draw (1.5,.75) node{\small $\pw{2,1,0}$};
\draw (0,1.5) node{\small $\pw{1,1,0}$};
\draw (-1,.25) node{\small $\bf{\pw{0,1,0}}$};
\draw (-3.25,-.5) node{\small $\bf{\pw{0,1,0}}$};
\draw (-2.5,1.5) node{\small $\pw{0,2,0}$};
\draw (-1,3.25) node{\small $\pw{1,2,0}$};
\draw (.4,3.6) node{\small $H_{21}^{a_1}$};
\draw (2.4,3.6) node{\small $H_{12}^{a_2}$};
\draw (2.4,0) node{\small $H_{23}^{c_1}$};
\draw (2.4,-3.6) node{\small $H_{13}^{b_1}$};

\draw (6.5,-3.464)--(10.5,3.464);
\draw (4.5,-3.464)--(8.5,3.464);
\draw (4.5,3.464)--(8.5,-3.464);
\draw (4.5,0)--(10.5,0);
\draw (6.5,-1.75) node{\small $\pw{0,0,0}$};
\filldraw [fill=black] (6.5,-.75) circle [radius=0.1];
\draw (7.5,-3.25) node{\small $\bf{\pw{1,0,0}}$};
\draw (7.5,-.35) node{\small $\bf{\pw{1,0,0}}$};
\draw (8.75,-1.75) node{\small $\pw{2,0,0}$};
\draw (10,.75) node{\small $\pw{2,1,0}$};
\draw (8,1) node{\small $\pw{1,1,0}$};
\draw (6.5,2) node{\small $\pw{1,2,0}$};
\draw (5,.75) node{\small $\pw{0,2,0}$};
\draw (5,-.75) node{\small $\pw{0,1,0}$};
\draw (8.9,3.6) node{\small $H_{21}^{a_1}$};
\draw (10.9,3.6) node{\small $H_{12}^{a_2}$};
\draw (8.9,-3.6) node{\small $H_{13}^{b_1}$};
\draw (10.9,0) node{\small $H_{23}^{c_1}$};
\end{tikzpicture}
\caption{In \cite{BB19} Baker shows that the graph on the left does not emit an arrangement with a bijective labeling despite satisfying conditions listed in Theorem \ref{bethany}. We illustrate this with two arrangements (center and right) corresponding to the graph. In the first arrangement (center) the label $\pw{0,1,0}$ appears twice, while in the second arrangement (right) the label $\pw{1,0,0}$ appears twice. One can modify the arrangements by changing the positive constants $a_1,a_2,b_1,$ and $c_1$ on the right hand sides of the equations of $H_{12}^{a_1},H_{12}^{a_2},H_{12}^{b_1},$ and $H_{12}^{c_1}$, but one cannot get rid of both dublicates at the same time (see \cite{BB19} for details).}\label{Figure: bethany example}
\end{center}
\end{figure}

\section{Central Affine Multigraphical Arrangements.}

\theoremstyle{definition}

In the case of central multigraphical arrangements, the arrangement is fully determined by the corresponding multigraph (up to a global shift). We start by characterizing the multigraphs corresponding to central arrangements.

\begin{theorem}\label{theorem: central linear arrangement}
Let $\mathcal{A}$ be a central multigraphical arrangement, then the corresponding multidigraph is simple and acyclic. Vice versa, if $G$ is a simple acyclic digraph, then there exists a central multigraphical arrangement $\mathcal{A},$ such that $\mathcal{G}_\mathcal{A}=G.$
\end{theorem}

\begin{proof}
Let $\mathcal{A}$ be a central multigraphical arrangement such that all hyperplanes intersect at the point $c=(c_1,c_2,\dots, c_n)$. Since all hyperplanes $H_{i,j}^a$ intersect at $c$, then we can have at most one $H_{i,j}^a$ for each pair $i,j$. Moreover, if we have a hyperplane $H_{i,j}^a$ then we cannot have a hyperplane of the form $H_{j,i}^b,$ because they would also be parallel. Thus the digraph $\mathcal{G}_\mathcal{A}$ is simple. \\
\indent Assume that $\mathcal{G}_\mathcal{A}$ contains the cycle ${i_0}\to {i_1}\to\cdots\to {i_k}\to {i_0}$. It then follows that the hyperplanes corresponding to the edges in the cycle exhibit
\begin{center}
\begin{tabular}{rcl}
$x_{i_0}-x_{i_1}$ &=&$a_{1}>0$\\
$x_{i_1}-x_{i_2}$ &=&$a_{2}>0$\\
$\vdots$\hspace{20pt} & $\vdots$ & $\hspace{5pt}\vdots$\\
$x_{i_{k-1}}-x_{i_k}$ &=&$a_{k}>0$\\
$x_{i_{k}}-x_{i_0}$ &=&$a_{k+1}>0$\\
\end{tabular}
\end{center}
Since each hyperplane passes through the point $c$ all these equations are satisfied at $\bf{x}=\bf{c}.$ After taking the sum of the above equations we see that $0=\sum_{i=1}^{k+1}a_{i}$ which contradicts the assumption that the $a_{i}>0$ for all $i$. Thus $\mathcal{G}_\mathcal{A}$ is acyclic.

Now, given an acyclic digraph $G=(V,E)$, with $V=\{1,\dots, n\}$, one can assume without loss of generality that the edges are oriented in an increasing way. We create the corresponding arrangement $\mathcal{A}$ by: for every edge $(i \to j)\in E$ create the hyperplane $H_{i,j}^{j-i}=\{x_i-x_j=j-i\}$. Consider the following point $c\in V$:

$$c=\left(\dfrac{n+1}{2},\dots,\dfrac{n+1}{2} \right)-(1,2,\dots, n)$$

We immediately see that the point $c$ lies in the intersection of all the hyperplanes since $c_i-c_j=j-i$ for all $1\le i<j\le n$. Therefore the graph $G$ has a corresponding central multigraphical arrangement. 
\end{proof}

Let $\mathcal{A}$ be a central multigraphical arrangement, and let $\mathcal{A}'$ be the linear arrangement obtained from $\mathcal{A}$ by shifting all the hyperplanes so that they pass through the origin. Let $G$ be the simple graph obtained from $\mathcal{G}_\mathcal{A}$ by removing the orientations on the edges. Then it is well-known that the acyclic orientations of $G$ are in one to one correspondence with the regions of $\mathcal{A}'$. The bijection is constructed as follows. Given a region $R$ of $\mathcal{A}'$ and an edge $i-j$ of $G,$ we orient it $i \to j$ if and only if $x_i<x_j$ at every point of $R$.

The regions of the original arrangement $\mathcal{A}$ are simply the regions of $\mathcal{A'}$ shifted by a vector. Therefore, they are also in bijection with the acyclic orientations of $G,$ or \textit{acyclic reorientations} of $\mathcal{G}_\mathcal{A}.$

\begin{theorem}\label{edge switch and hyperplane crossing}
The fundamental region of $\mathcal{A}$ corresponds to the original orientation of $\mathcal{G}_\mathcal{A}$, and crossing a hyperplane $H_{i,j}^a\in\mathcal{A}$ switches the orientation of the corresponding edge between $i$ and $j.$
\end{theorem}
\begin{proof}
Let $R_0$ be the fundamental region of the arrangement $\mathcal{A},$ and let $\mathcal{A}'$ be the corresponding linear arrangement. Let $c=(c_1,\dots, c_n)$ be in the intersection of all the hyperplanes of the arrangement $\mathcal{A}$.Then it follows that $-c$ belongs to the corresponding region $R'=R_0-c$ of $\mathcal{A}'$. Therefore, if $H_{i,j}^a\in\mathcal{A}$ and the edge $i\to j$ is the corresponding edge in $\mathcal{G}_\mathcal{A}$, then at $c$ we have $c_i-c_j=a$, in particular we have that $c_i>c_j$. It then follows that at $-c\in R'$ that we have $-c_i<-c_j$. Thus, in the orientation corresponding to $R'$ we also get the edge oriented as $i\to j$.

Finally, crossing a hyperplane $H_{i,j}^a$ corresponds to crossing the hyperplane $x_i=x_j$ of the linear arrangement $\mathcal{A'},$ which switches the orientation of the corresponding edge.
\end{proof}

%\begin{theorem}
%Given a central multigraphical arrangement $\mathcal{A}$, then  acyclic reorientations of the graph $\mathcal{G}_\mathcal{A}$ correspond to the regions of the arrangement $\mathcal{A}$. More precisely, the fundamental region corresponds to the original orientation, $\mathcal{G}_\mathcal{A}$, and crossing the hyperplane $H_{i,j}^a\in\mathcal{A}$ switches the orientation of the corresponding edge in $\mathcal{G}_\mathcal{A}$.
%\end{theorem}

%\begin{proof}
%The proof follows from the previous lemma.
%\end{proof}

%
%Starting from the fundamental region with corresponding graph $G_0$, we are able to create an acyclic reorientations $G'$ by keeping track of the hyperplanes $H_{i,j}$ as you move away from the origin. For every hyperplane $H_{i,j}$ that is crossed, we replace the edge $i \to j$ with $j\to i$. \\

%calculate g-parking example

\begin{figure}  
\begin{center}
\begin{tikzpicture}[scale=.5]
%Hyperplane Grid and Labels
\filldraw [fill=black] (2,-2) circle [radius=0.1];
\draw (-3,0)--(11,0);  %horizontal line
\draw (1,-5.196)--(7,5.196);   %/ line
\draw (1,5.196)--(7,-5.196);   %\ line
\draw (7.75,5.196) node{\small $H_{1,2}^a$};
\draw (7.2,-5.5) node{\small $H_{1,3}^b$};
\draw (11.75,0) node{\small $H_{2,3}^c$};
%graph 0
\draw (10,3) node{\small $\pw{2,1,0}$};
\draw (7,3.53) node{\small $1$};
\draw (5.53,1.196) node{\small $2$};
\draw (8.47,1.196) node{\small $3$};
\draw (18,3) node{\small $x_3<x_2<x_1$};
\draw [-{Latex[length=2mm, width=2mm]}] (13,2)--(19,5);
\draw [-{Latex[length=2mm, width=2mm]}] (8.45,1.5)--(7.2,3.3);
\draw [-{Latex[length=2mm, width=2mm]}] (8.2,1.1)--(5.8,1.1);
\draw [-{Latex[length=2mm, width=2mm]}] (5.53,1.5)--(6.8,3.3);
%graph 1
\draw (10,-4.5) node{\small $\pw{2,0,0}$};
\draw (8,-1) node{\small $1$};
\draw (6.5,-3) node{\small $2$};
\draw (9.5,-3) node{\small $3$};
\draw [-{Latex[length=2mm, width=2mm]}] (9.3,-2.7)--(8.3,-1.3);
\draw [-{Latex[length=2mm, width=2mm]}] (6.8,-3)--(9.2,-3);
\draw [-{Latex[length=2mm, width=2mm]}] (6.8,-2.7)--(7.8,-1.3);
%graph 2
\draw (4,-6) node{\small $\pw{1,0,0}$};
\draw (4,-1.53) node{\small $1$};
\draw (2.53,-4.196) node{\small $2$};
\draw (5.47,-4.196) node{\small $3$};
\draw [-{Latex[length=2mm, width=2mm]}] (4.2,-1.83)--(5.46,-3.796);
\draw [-{Latex[length=2mm, width=2mm]}] (2.53,-3.796)--(3.9,-1.83);
\draw [-{Latex[length=2mm, width=2mm]}] (2.73,-3.896)--(5.26,-3.896);
%graph 3
\draw (-1,-4.5) node{\small $\pw{0,0,0}$};
\draw (-8,-4) node{\small $x_1<x_2<x_3$};
\draw (0,-1) node{\small $1$};
\draw (-1.47,-3) node{\small $2$};
\draw (1.47,-3) node{\small $3$};
\draw [-{Latex[length=2mm, width=2mm]}] (-5,-4)--(-11,
-7);
\draw [-{Latex[length=2mm, width=2mm]}] (-1.17,-3)--(1.17,-3);
\draw [-{Latex[length=2mm, width=2mm]}] (-.3,-1.3)--(-1.17,-2.7);
\draw [-{Latex[length=2mm, width=2mm]}] (.3,-1.3)--(1.17,-2.7);
%graph 40,1
\draw (-1,4.5) node{\small $\pw{0,1,0}$};
\draw (0,3.53) node{\small $1$};
\draw (-1.47,1.196) node{\small $2$};
\draw (1.47,1.196) node{\small $3$};
\draw [-{Latex[length=2mm, width=2mm]}] (1.17,1.196)--(-1.17,1.196);
\draw [-{Latex[length=2mm, width=2mm]}] (-.3,3.3)--(-1.3,1.6);
\draw [-{Latex[length=2mm, width=2mm]}] (.3,3.3)--(1.47,1.5);
%graph 5
\draw (4,6.5) node{\small $\pw{1,1,0}$};
\draw (4,5.5) node{\small $1$};
\draw (2.53,3.5) node{\small $2$};
\draw (5.47,3.5) node{\small $3$};
\draw [-{Latex[length=2mm, width=2mm]}] (5.17,3.5)--(2.8,3.5);
\draw [-{Latex[length=2mm, width=2mm]}] (5.5,3.8)--(4.3,5.3);
\draw [-{Latex[length=2mm, width=2mm]}] (3.7,5.2)--(2.7,3.8);
\end{tikzpicture}
\end{center}
\caption{We consider the central arrangement corresponding to the digraph $\mathcal{G}_\mathcal{A}=(1\to 2,1 \to 3,2\to 3).$ The regions of the arrangement are labeled by the corresponding reorientations and the generalized Pak-Stanley labels. Note that the fundamental region is labeled by $\mathcal{G}_\mathcal{A}$ and $\pw{0,0,0},$ and as we cross the hyperplanes the orientations of the corresponding edges switch. Moreover, as we cross the hyperplane $H_{i,j}^a$ in a direction away from the origin, the $i$th entry of the Pak-Stanley label increases by 1.}\label{Figure: arrangement with labels}
\end{figure}

\begin{lemma}
The Pak-Stanley labels for the arrangement $\mathcal{A}$ can be computed in terms of acyclic reorientations of the graph $\mathcal{G}_\mathcal{A}$. More precisely, for a region $R$ of $ \mathcal{A}$ the label $\lambda_R(i)$ equals to the number of edges of $\mathcal{G}_\mathcal{A}$ leading from $i,$ such that their orientations got switched in the reorientation corresponding to $R$.
\end{lemma}
\begin{proof}
For an arrangement $\mathcal{A}$ the Pak-Stanley label for a region $R$ is calculated by counting the number of hyperplanes of the form $H_{i,j}^a$ separating $R$ from the origin and increasing the value $\lambda_R(i)$ accordingly. However, Theorem \ref{edge switch and hyperplane crossing} implies that as we cross a hyperplane $H_{i,j}^a$ we reorient the edge from $(i\to j)$ to $(j\to i)$, so it follows that $\lambda_R(i)$ is the number of edges of $\mathcal{G}_\mathcal{A}$ leading from $i$ that get reoriented in the graph corresponding to $R$.
\end{proof}

%{\color{red} Mikhail: I suggest moving the example explanations in the caption of the Figure. Also, I think it is unnecessary to talk about crossing a particular hyperplane. Should be enough to simply say that as we cross a hyperplane $H_{i,j}^a$ the corresponding edge changes the orientation, and if we are moving in the direction away from the origin, then the $i$th label increases by one.}

Now we are ready to prove our main theorem:

\begin{theorem}\label{Theorem: Main}
Let $V=\{1,2,\dots, n\}$ and $G=(V,E)$ be an acyclic directed graph on $n$ vertices with edges oriented in the increasing way. Then the hyperplane arrangement corresponding to $G$ produces duplicate Pak-Stanley labelings if and only if there exists $1\le k <i<j\le n$ such that $(k\to i),(k\to j)\in E$ and $(i\to j)\notin E$.
\end{theorem}

\begin{proof}[Proof of Theorem \ref{Theorem: Main}]

$\Rightarrow)$ Assume that $G$ produces duplicate Pak-Stanley labelings and for the sake of contraction assume that no such $i,j,k$ exists. Since labelings correspond to acyclic reorientations of $G$, let $G'=(V,E')$ and $G''=(V,E'')$ be such reorientations. \\
\indent Since reorientations are in correspondence with labelings then  there is an edge $k\to i$ of $\mathcal{G}_\mathcal{A}$ that is reoriented as $i \to k$ in $G'$ but not in $G''$. Moreover since the labels are equal, then there must also be another edge emanating from $k$, say edge $k\to j$, such that it is reoriented as $j \to k$ in $G''$ but not in $G'$.  In other words, the duplicate labeling implies that we have edges $(i\to k),(k\to j)\in E'$ and $(k\to i),(j\to k)\in E''$.\\
\indent Let $k$ be the largest integer such that this occurs. Since $k$ is the largest possible, it follows that all edges between vertices $p,q$ where $p,q>k$ are oriented in the same way in both reorientations. Without loss of generality we can assume that $i<j$. This gives arise to two cases depending on whether or not the edge from $i \to j$,  is oriented as $i \to j$ or $j\to i$ in both $G'$ and $G''$. If we have the edge $i \to j$ then in $G''$ we have the cycle $k \to i\to j \to k$, a contradiction since $G$-parking functions arise from acyclic reorientations. Otherwise we have the edge $j\to i$, but as before we have the cycle $k \to j\to i\to k$ in $G'$ (see Figure \ref{Figure: cycles in reorientations}).

\begin{figure}\label{Figure: cycles in reorientations}
\begin{center}
\begin{tikzpicture}   

\draw (15, -6) node {\large $G''$};
\draw (11,-8.3) node {\small $k$};
\draw (12.5,-8.3) node {\small $\dots$};
\draw (14,-8.3) node {\small $i$};
\draw (15,-8.3) node {\small $\dots$};
\draw (16,-8.3) node {\small $j$};

\draw [][-{Latex[red,length=2mm, width=2mm]}] (16,-7.7) .. controls (13.5,-6) .. (11.1,-7.8);
\draw [-{Latex[length=2mm, width=2mm]}] (11.1,-8) .. controls (12.5,-7) .. (14,-8);
\draw [-{Latex[length=2mm, width=2mm]}] (14.1,-8) .. controls (15,-7.5) .. (16,-8);
%\draw [-{Latex[length=2mm, width=2mm]}] (16,-8) .. controls (15,-7.5) .. (14.1,-8);

\draw (2, -6) node {\large $G'$};
\draw (1,-8.3) node {\small $k$};
\draw (2.5,-8.3) node {\small $\dots$};
\draw (4,-8.3) node {\small $i$};
\draw (5,-8.3) node {\small $\dots$};
\draw (6,-8.3) node {\small $j$};

\draw [-{Latex[length=2mm, width=2mm]}] (1.1,-7.8) .. controls (3.5,-6) .. (6,-7.7);
\draw [][-{Latex[red,length=2mm, width=2mm]}] (4,-8) .. controls (2.5,-7) .. (1.1,-8);
\draw [-{Latex[length=2mm, width=2mm]}] (6,-8).. controls (5,-7.5) .. (4.1,-8);

%\draw (6,-12.3) node {\small $v_k$};
%\draw (8,-12.3) node {\small $\dots$};
%\draw (10,-12.3) node {\small $v_i$};
%\draw (11,-12.3) node {\small $\dots$};
%\draw (12,-12.3) node {\small $v_j$};

%\draw [arrows={-angle 60}] (6.1,-12) .. controls (9,-10) .. (11.9,-12);
%\draw [arrows={-angle 60}] (6.1,-12) .. controls (8,-11) .. (9.9,-12);

\end{tikzpicture}
\end{center}
\caption{Here we see the two reorientations of the graph $G$, $G'$ and $G''$, and the corresponding cycles created depending on the orientation of the edge $i\to j$.}%\label{Figure: cycles in reorientations}
\end{figure}

%Assume there exists a triple $k<i<j$ such that $(k\to i),(k\to j)\in E$ and $
%(i\to j)\notin E$. We will explicitely construct $G'=(V,E')$ and $G''=(V,E'')$ such that they produce the same label. Consider the acyclic reorientation of $G$ where edges eminating from vertex $k$ to vertices ${k+1},{k+2},\dots,i-1,{i+1},\dots, {j-1}$ are reversed as well as edges eminating from $i$ to vertices ${i+1},{i+2},\dots, {j-1}$. In addition, for the reorientation $G'=(V,E')$ we switch the orientation of $k\to i$ while for $G''=(V,E'')$ we switch the orientation of $k\to j$. This produces the duplicate labeling \[\tau=\pw{0,\dots,0, \overset{k\text{th}}{(j-k-1)},0,\dots,0, \overset{i\text{th}}{(j-i-1)},0,\dots,0}\]Therefore $G$ produces duplicate labelings. 

$\Leftarrow)$ The easiest way to produce the acyclic reorientations, $G'$ and $G'',$ is reordering the vertices and reorienting the edges so that they point in the increasing direction after considering the new vertex order. For the reoriented graph $G'=(V,E')$ we reorder the vertices of $G'$ as follows $$1\prec \dots \prec k-1\prec k+1\prec \dots \prec i-1\prec i+1\prec\dots\prec j-1\prec i\prec k\prec j\prec \dots\prec n. $$
In other words, for $G'$ we move the vertices $k+1, \dots,i-1,i+1,\dots,j-1$ to the left so that they precede vertex $k$, and then switch vertices $k$ and $i$. Note that as we reorder the vertices, the only edges that are reversed are
\begin{center}
    \begin{tabular}{lll}
    (1) &$(k\to p)\in E$  such that :& \\
   & & $p\in\{k+1,\dots, i-1\}$, or\\
   & & $p\in\{i+1,\dots, j-1\}$, or\\
   && $p=i$\\\\
    (2) &$(i\to p)\in E$ such that: & $p\in\{i+1,\dots, j-1\}$.
    \end{tabular}
\end{center}
To produce the reorientation that corresponds to $G''=(V,E'')$ we reorder the vertices of $G''$ as follows:
$$1\prec \dots \prec k-1\prec k+1\prec\dots\prec i-1\prec i+1\prec \dots\prec j-1\prec j\prec k\prec i\prec j+1\prec \dots\prec n.$$ 
In other words, for $G''$ we move the vertices $k+1,\dots, i-1,i+1,\dots, j-1$ so that they precede vertex $k$, but now we move vertex $j$ two places to the left so that it precedes $k$ instead of switching vertices $k$ and $i$. This time the following edges are reoriented
\begin{center}
    \begin{tabular}{lll}
    (1) &$(k\to p)\in E$  such that :& \\
   & & $p\in\{k+1,\dots, i-1\}$, or\\
   & & $p\in\{i+1,\dots, j-1\}$, or\\
   && $p=j$\\\\
    (2) &$(i\to p)\in E$ such that: & $p\in\{i+1,\dots, j-1\}$.
    \end{tabular}
\end{center}
Note that $(i\to j)\notin E$ by assumption, therefore it does not need to be reoriented. \\
We conclude that both $G'=(V,E')$ and $G''=(V,E'')$ produce the labeling
\[\tau=\pw{0,\dots,0, \overset{k\text{th}}{(N+1)},0,\dots,0, \overset{i\text{th}}{(K)},0,\dots,0}\]
where \[N=\#\{(k\to p)\in E: p\in\{k+1,\dots, i-1\}\cup\{i+1,\dots, j-1\}\}\]and \[K=\#\{(i\to p)\in E: p\in\{i+1,\dots, j-1\}\}.\]
\end{proof}

%%EXAMPLE OF THEOREM

\begin{example}
Consider the following graph $G=(V,E)$ where the vertex and edge sets are given by $V=\{1,2,3,4\}$ and $E=\{(1\to 2),(1\to 3),(1\to 4),(2\to 3),(2\to 4)\}$.

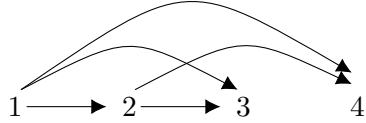
\begin{center}
\begin{tikzpicture}[scale=.75]
\draw (11,-8.3) node {\small $1$};
\draw (13,-8.3) node {\small $2$};
\draw (15,-8.3) node {\small $3$};
\draw (17,-8.3) node {\small $4$};

\draw [-{Latex[length=2mm, width=2mm]}] (11.1,-8) .. controls (14,-6) .. (16.9,-7.7);
\draw [-{Latex[length=2mm, width=2mm]}] (11.1,-8) .. controls (13,-7) .. (14.9,-8);
\draw [-{Latex[length=2mm, width=2mm]}] (13.1,-8) .. controls (15,-7) .. (16.9,-7.9);
\draw [-{Latex[length=2mm, width=2mm]}] (11.2,-8.3)--(12.6,-8.3);
\draw [-{Latex[length=2mm, width=2mm]}] (13.2,-8.3)--(14.6,-8.3);
%\draw [arrows={-angle 60}] (15.2,-8.3)--(16.8,-8.3);

\end{tikzpicture}

\end{center}

In this example we see that $(1\to 3)$ and $(1\to 4)$, but $(3\to 4)\notin E$, so Theorem \ref{Theorem: Main} implies that there should exist two reorientations $G'$ and $G''$ that produce the same Pak-Stanley labeling. Consider the following reorientations

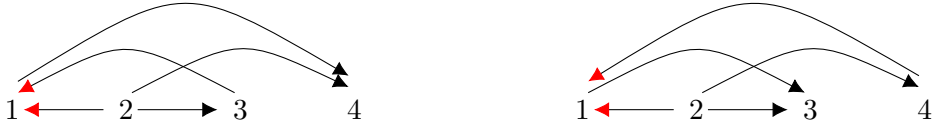
\begin{center}
\begin{tikzpicture}[scale=.75]
\draw (11,-8.3) node {\small $1$};
\draw (13,-8.3) node {\small $2$};
\draw (15,-8.3) node {\small $3$};
\draw (17,-8.3) node {\small $4$};

\draw [-{Latex[red,length=2mm, width=2mm]}] (16.9,-7.7) .. controls (14,-6) .. (11.1,-7.8);
\draw [-{Latex[length=2mm, width=2mm]}] (11.1,-8) .. controls (13,-7) .. (14.9,-8);
\draw [-{Latex[length=2mm, width=2mm]}] (13.1,-8) .. controls (15,-7) .. (16.9,-7.9);
\draw [-{Latex[red,length=2mm, width=2mm]}] (12.6,-8.3)--(11.2,-8.3);
\draw [-{Latex[length=2mm, width=2mm]}] (13.2,-8.3)--(14.6,-8.3);
%\draw [arrows={-angle 60}] (15.2,-8.3)--(16.8,-8.3);

\draw (1,-8.3) node {\small $1$};
\draw (3,-8.3) node {\small $2$};
\draw (5,-8.3) node {\small $3$};
\draw (7,-8.3) node {\small $4$};

\draw [-{Latex[length=2mm, width=2mm]}] (1.1,-7.8) .. controls (4,-6) .. (6.9,-7.7);
\draw [-{Latex[red,length=2mm, width=2mm]}] (4.9,-8) .. controls (3,-7) .. (1.1,-8);
\draw [-{Latex[length=2mm, width=2mm]}] (3.1,-8) .. controls (5,-7) .. (6.9,-7.9);
\draw [-{Latex[red,length=2mm, width=2mm]}] (2.6,-8.3)--(1.2,-8.3);
\draw [-{Latex[length=2mm, width=2mm]}] (3.2,-8.3)--(4.6,-8.3);

\end{tikzpicture}

\end{center}

These two reorientations of $G_{\mathcal{A}}$ produce the label $\pw{2,0,0,0}$. Similarly for $(2\to 3),(2\to 4)\in E$, but $(3\to 4)\notin E$ there will be duplicates

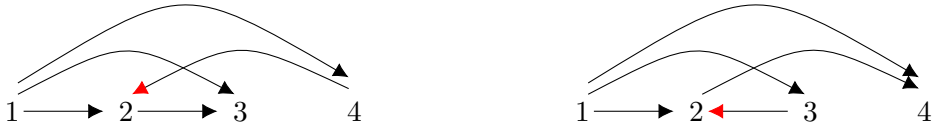
\begin{center}
\begin{tikzpicture}[scale=.75]
\draw (11,-8.3) node {\small $1$};
\draw (13,-8.3) node {\small $2$};
\draw (15,-8.3) node {\small $3$};
\draw (17,-8.3) node {\small $4$};

\draw [-{Latex[length=2mm, width=2mm]}] (11.1,-7.8) .. controls (14,-6) .. (16.9,-7.7);
\draw [-{Latex[length=2mm, width=2mm]}] (11.1,-8) .. controls (13,-7) .. (14.9,-8);
\draw [-{Latex[length=2mm, width=2mm]}] (13.1,-8) .. controls (15,-7) .. (16.9,-7.9);
\draw [-{Latex[length=2mm, width=2mm]}] (11.2,-8.3)--(12.6,-8.3);
\draw [-{Latex[red,length=2mm, width=2mm]}] (14.6,-8.3)--(13.2,-8.3);
%\draw [arrows={-angle 60}] (15.2,-8.3)--(16.8,-8.3);

\draw (1,-8.3) node {\small $1$};
\draw (3,-8.3) node {\small $2$};
\draw (5,-8.3) node {\small $3$};
\draw (7,-8.3) node {\small $4$};

\draw [-{Latex[length=2mm, width=2mm]}] (1.1,-7.8) .. controls (4,-6) .. (6.9,-7.7);
\draw [-{Latex[length=2mm, width=2mm]}] (1.1,-8) .. controls (3,-7) .. (4.9,-8);
\draw [-{Latex[red,length=2mm, width=2mm]}] (6.9,-7.9) .. controls (5,-7) .. (3.1,-8);
\draw [-{Latex[length=2mm, width=2mm]}] (1.2,-8.3)--(2.6,-8.3);
\draw [-{Latex[length=2mm, width=2mm]}] (3.2,-8.3)--(4.6,-8.3);

\end{tikzpicture}

\end{center}

These two reorientations of $G$ produce the duplicate label $\pw{0,1,0,0}$. Actually, this graph produces four more duplicate labelings \[\{\pw{1,1,0,0}, \pw{2,1,0,0}, \pw{1,2,0,0}, \pw{3,1,0,0}\}.\]
\end{example}

%\begin{corollary}
%Given a multigraphical hyperplane arrangement $\mathcal{A}$ with an injective map $\varphi: R(\mathcal{A})\rightarrow  \ell(\mathcal{A})$, from the set of regions of $\mathcal{A}$ to the set of labels of $\mathcal{A}$. Then for an central subarrangement $\mathcal{A}'\subset \mathcal{A}$, $\varphi': R(\mathcal{A}')\rightarrow \ell(\mathcal{A}')$ is injective. 
%\end{corollary}

\end{document}